\documentclass[11pt]{35}
\usepackage{amsthm, amsmath, amssymb, amsbsy, amsfonts, latexsym, euscript}
\usepackage{calrsfs} 
\usepackage{graphicx}

\DeclareMathAlphabet{\varmathbb}{U}{pxsyb}{m}{n}

\def\leq{\leqslant}

\def\phi{\varphi}

\def\kappa{\varkappa}

\newcommand{\D}{\mathrm{d}\kern0.2pt}%

\newcommand{\E}{\mathrm{e}\kern0.2pt} 
\newcommand{\ii}{\kern0.05em\mathrm{i}\kern0.05em}

\newcommand{\RR}{\mathbb{R}}%

\newtheorem{theorem}{\bf \indent Theorem}[section]
\newtheorem{proposition}{\bf \indent Proposition}[section]

\theoremstyle{remark}

\numberwithin{equation}{section}

\begin{document}

\noindent {\Large \bf  The fundamental eigenfrequency is simple \\[3pt] in the
two-dimensional sloshing problem}

\vskip5mm

{\bf Nikolay Kuznetsov}

\vskip-2pt {\small Laboratory for Mathematical Modelling of Wave Phenomena}
\vskip-4pt {\small Institute for Problems in Mechanical Engineering} \vskip-4pt
{\small Russian Academy of Sciences} \vskip-4pt {\small V.O., Bol'shoy pr. 61, St.
Petersburg, 199178} \vskip-4pt {\small Russian Federation} \vskip-4pt {\small
nikolay.g.kuznetsov@gmail.com}

\vskip4mm

\parbox{134mm} {\noindent The two-dimensional sloshing problem is considered; it
describes the transversal free oscillations of water in an open, infinitely long
canal of uniform cross-section. It is proved that the fundamental eigenfrequency is
simple, whereas the corresponding velocity potential has only one nodal line
connecting the free surface and the bottom; its harmonic conjugate (stream function)
does not change sign under the proper choice of the additive constant.

\vspace{2mm}


\vspace{2mm}

}

\section{Introduction and the main result}

\noindent This paper deals with a spectral problem usually referred to as the
two-dimensional sloshing problem; it describes the frequencies and modes of the
transversal free oscillations of water in an infinitely long canal having a uniform
cross-section.

\subsection{Statement of the sloshing problem}
 
Let the canal's cross-section $W \subset \RR^2_- = \{ (x,y) \in \RR^2: y < 0 \}$ be
a bounded, simply connected domain, whose piecewise smooth boundary $\partial W$
has no cusps. One of the open arcs forming $\partial W$ is an interval $F$ of the
$x$-axis (the free surface of water), and the bottom $B = \partial W\setminus
\overline F$ is the union of open arcs lying in $\RR^2_-$ complemented by corner
points (if there are any) connecting these arcs.

With a time-harmonic factor removed, the velocity potential $u(x,y)$ of the sloshing
motion satisfies the mixed Steklov--Neumann boundary value problem:
\begin{eqnarray}
&& u_{xx} + u_{yy} = 0\quad {\rm in}\ W, \label{lap} \\ && u_y =
\nu u\quad {\rm on}\ F, \label{nu} \\ && \partial u/\partial n =
0\quad {\rm on}\ B. \label{nc}
\end{eqnarray}
Here $\partial /\partial n$ stands for the exterior normal derivative on $B$. It is
convenient to complement this problem with the orthogonality condition
\begin{equation}
\int_F u (x, 0) \, \D x = 0 , \label{ort}
\end{equation}
thus excluding the zero eigenvalue of \eqref{lap}--\eqref{nc}. Then the hydrodynamic
interpretation of the spectral parameter $\nu > 0$ is $\nu = \omega^2/g$, where
$\omega$ is the radian frequency of the water oscillations and $g$ is the
acceleration due to gravity.

In what follows, problem \eqref{lap}--\eqref{ort} is called the {\it 2D sloshing
problem}; its statement given above is not the most general one, but it commonly
used in applications. Over more than two centuries, the problem has been the subject
of a great number of studies; a historical review can be found in \cite{FK}. Since
the 1950s, it is known that this problem has a discrete spectrum; that is, there
exists a sequence of eigenvalues
\begin{equation*}
0<\nu_1 \leq \nu_2 \leq \dots \leq \nu_n \leq \dots ,
\label{seq}
\end{equation*}
each counted according to its multiplicity, and it tends to infinity as $n \to
\infty$; see, for example,~\cite[\S~3.3.4]{KK}.

\subsection{Background and the main result}

Exact solutions of the 2D sloshing problem are known for a few geometries (a list is
given in \cite{FK}); most of them may be obtained either by separation of variables
or by an inverse procedure which seeks a region associated with a specified
eigenfunction (see, for example, the recent note \cite{W}). All eigenvalues are
simple for each of these geometries. Another problem for which the simplicity of all
eigenvalues was proved is the 2D ice-fishing problem (see \cite{KM}), wherein the
water domain $W = \RR^2_-$ is covered by the rigid ice along $B = \partial \RR^2_-
\setminus \overline F$, whereas the free surface $F$ is either $\{ |x| < 1, y=0 \}$
(a single gap in the ice) or $\{ b < |x| < b+1, y=0 \}$ (two gaps at the spacing $2
b > 0$).

The recent immense article \cite{SS}, concerning the sharp spectral asymptotics for
the 2D sloshing problem in domains with corners, shows that it would be extremely
difficult to establish the simplicity of all sloshing eigenvalues in the case of a
general water domain. Therefore, our aim is more modest; namely, to prove the
following.

\begin{theorem}
{\rm(i)} The fundamental eigenvalue $\nu_1$ of problem \eqref{lap}--\eqref{ort} is
simple. {\rm(ii)} The corresponding eigenfunction $u_1$ has only one nodal line
connecting $F$ and $\overline B$.
\end{theorem}

It turns out that this assertion is still an open question for an arbitrary bounded
$W$ despite several attempts to resolve it. The first one was made by Kuttler
\cite{Kut}, whose proof used the following fallacious lemma:
\begin{quote}
{\it Nodal lines of an eigenfunction of problem \eqref{lap}--\eqref{ort} have one
end on the free surface and the other one on the bottom.}
\end{quote}
A counterexample to this assertion was constructed in the paper \cite{KKM}, in which
the authors tried to give their own proof of Theorem~1.1. As is shown in the next
section, their approach actually leads to the result, but, unfortunately, an
incorrect inequality was applied in \cite{KKM} on the final stage of the proof. The
defect was corrected in \cite{KKM1}, but at the expense of an extra assumption
imposed on $W$. Namely, it must satisfy John's condition confining $W$ to the strip
between vertical lines drown through the endpoints of $\overline F$.

\section{Proof of Theorem 1.1}

Since $u_y - \nu u$ vanishes on $F$, its extension across $F$ as an odd function of
$y$ is harmonic. This yields a representation of $u (x, y)$ valid on both sides of
$F$; see \cite[p.~95]{John1} for details. Therefore, the Cauchy--Riemann equations
for $u$ and the stream function $v$ (a harmonic conjugate of $u$ in $W$) are valid
on $F$ as well; moreover, these functions are differentiable along $F$.

Following the approach proposed in \cite{KKM}, let us consider an equivalent to
\eqref{lap}--\eqref{ort} spectral problem for $v$. Indeed, the Cauchy--Riemann
equations reduce \eqref{lap}--\eqref{ort} to
\begin{eqnarray}
&& v_{xx} + v_{yy} = 0 \quad {\rm in} \ W , \label{lapv} \\ && -v_{xx} = \nu v_y
\quad {\rm on} \ F , \label{nuv} \\ && v = 0 \quad {\rm on} \ B , \label{dcv}
\end{eqnarray}
and vice versa. Notice that obtaining condition \eqref{dcv} requires also an
appropriate choice of the additive constant. Besides, it implies both conditions
\eqref{nc} and \eqref{ort}. It is obvious that all eigenvalues of problems
\eqref{lapv}--\eqref{dcv} and \eqref{lap}--\eqref{ort} have the same multiplicity.

\subsection{Variational principle for the stream function}

The variational principle for problem \eqref{lapv}--\eqref{dcv} was proposed in
\cite{KKM}, but it involves nonlocal operators which, presumably, entails its
disadvantage. Prior to introducing a local variational principle, it is convenient
to consider an appropriate weak formulation of problem \eqref{lapv}--\eqref{dcv}. To
derive the requisite integral identity we write the first Green's identity for $v$
and transform it with the help of \eqref{nuv} and integration by parts
\begin{equation}
\int_W |\nabla v|^2 \D x \, \D y = \int_F v_y \, v \, \D x = - \nu^{-1} \int_F
v_{xx} \, v \, \D x = \nu^{-1} \int_F v_{x}^2 \, \D x \, , \label{Gr}
\end{equation}
where the integrated term vanishes in view of \eqref{dcv}; here and below $\nabla v
= (v_x, v_y)$.

Thus, it is reasonable to seek a weak solution in $\mathcal H = H^1_B (W) \cap H^1_0
(F)$. Here $H^1_B (W)$ is the subspace of the Sobolev space $H^1 (W)$ consisting of
functions that vanish on~$B$ (see \newline \cite[\S 7.1]{A} for details). An
equivalent norm in $H^1_B (W)$ is equal to the integral on the left-hand side of
\eqref{Gr}; moreover, $H^1_B (W)$ is isomorphic to $H^{1/2} (F)$. By $H^1_0 (F)$ we
denote the closure of smooth, compactly supported on $F$ functions in the $H^1 (F)$
norm; therefore, an equivalent norm in $H^1_0 (F)$ is equal to the integral on the
right-hand side of \eqref{Gr}.

Thus, a weak solution of problem \eqref{lapv}--\eqref{dcv} is $v \in \mathcal H$ if
the following integral identity
\[ \nu \int_W \nabla v \cdot \nabla \psi \, \D x \, \D y = \int_F v_x \, \psi_x \, \D x
\]
holds for an arbitrary $\psi \in \mathcal H$. This suggests the following
variational principle
\begin{equation}
\nu_1 = \min_{w \in \mathcal H} \ \frac{\int_F w_x^2 \, \D x}{\int_W |\nabla w|^2 \,
\D x \D y} \, . \label{vp}
\end{equation}
for the fundamental eigenvalue $\nu_1$ of this problem. It was mentioned that the
quadratic form in the denominator is equivalent to the norm in $H^{1/2} (F)$. Since
$H^1_0 (F)$ is compactly embedded into the latter space, there exists a nontrivial
$w^*$, which delivers minimum to to the variational quotient \eqref{vp}. Moreover,
it is easy to verify that $w^*$ satisfies problem \eqref{lapv}--\eqref{dcv} with
$\nu = \nu_1$.

\subsection{Auxiliary results and proof of Theorem 1.1}

Let $N (v) = \{ (x,y) \in \overline W: \, v (x,y)=0 \}$ denote the set of nodal
lines of a sloshing eigenfunction $v$. A connected component of $W \setminus N$ is
called a nodal domain of $v$. A key assertion for our considerations is the following
analogue of the Courant nodal domain theorem.

\begin{proposition}
Any stream eigenfunction corresponding to the eigenvalue $\nu_1$ has a single nodal
domain.
\end{proposition}

\begin{proof}
Assuming the contrary, we denote by $W'$ and $W''$ two nodal domains of an
eigenfunction $v$. Each of them is bounded above by a single subinterval of $F$, say
$F'$ and $F''$, respectively. Otherwise, the trace $v (x, 0)$ has two zeros inside
$F$ (as well as in the case of three nodal domains), and so there are three critical
points, where $v_x (x, 0)$ changes sign. Then the Cauchy--Riemann equations and
condition \eqref{nu} imply that the harmonic conjugate $u (x, 0)$ also changes sign
three times. But this is impossible because it is known that $u (x, 0)$ changes sign
not more than two times; see \cite[Corollary~2.9]{KKM}.

Let us define $\psi$ on $\overline W$ as follows: $\psi = v$ on $\overline{W'}$ and
$\psi = 0$ elsewhere. It is clear that $\psi \in \mathcal H$. Similarly to
\eqref{Gr} we have:
\begin{equation*}
\int_W |\nabla \psi|^2 \D x \, \D y = \int_{W'} |\nabla v|^2 \D x \, \D y =
\int_{F'} v_y \, v \, \D x = \nu^{-1}_1 \int_{F'} v_{x}^2 \, \D x = \nu^{-1}_1
\int_F \psi_{x}^2 \, \D x .
\end{equation*} 
Hence $\psi$ delivers minimum to the variational quotient \eqref{vp}, and so it is
an eigenfunction of problem \eqref{lapv}--\eqref{dcv} corresponding to $\nu_1$.
However, this contradicts to the unique continuation property of harmonic functions
in view of the definition of $\psi$.
\end{proof}

This proof also implies the following.

\begin{proposition}
For any stream eigenfunction corresponding to $\nu_1$, its trace on $F$ cannot
change sign; moreover, it has a single extremum.
\end{proposition}

\begin{proof}[Proof of Theorem 1.1]
(i) Let us assume the existence of two linearly independent stream eigenfunctions
corresponding to $\nu_1$, and obtain a contradiction from this assumption.
Proposition~2.1 allows us to suppose that these functions, say $v'$ and $v''$, are
positive; notice that this does not contradict the orthogonality condition $\int_W
\nabla v' \cdot \nabla v'' \, \D x \, \D y = 0$. Moreover, Proposition~2.2 implies
that each of these functions has a single extremum on~$F$; namely, maximum.

Let $M' \ (M'')$ denote the maximum value of $v'$ ($v''$, respectively) attained at
the point $(x', 0)$ ($(x'', 0)$, respectively). Consider $V (x) = M'' v' (x, 0) - M'
v'' (x, 0)$, which cannot vanish identically on $F$ even if $x' = x''$. In this
case, $V$ has three zeros on $\overline F$, and so at least two extrema, but this
contradicts Proposition~2.2. Finally, if $x' \neq x''$, then $V$ changes sign, which
is also impossible by Proposition~2.2.

(ii) Kuttler's reasoning (see \cite[p. 1236]{Kut}) turns out to be correct provided
the unnecessary reference to the fallacious lemma is omitted. Indeed, it is a
version of the original proof by Courant (see \cite[p.~452]{CH}), which implies that
$u_1$ does not have more than two nodal domains. Then condition \eqref{ort} yields
that two such domains really exist; the nodal line separating them has only one end
on $F$ in view of the second assertion of Proposition~2.2. Thus the proof is
complete.
\end{proof}

\vspace{-12mm}

\renewcommand{\refname}{
\begin{center}{\Large\bf References}
\end{center}}


\begin{thebibliography}{99}

\vspace{-2mm}

{\small

\bibitem{A} J.-P. Aubin, {\it Approximation of Elliptic Boundary-Value Problems.}
Wiley-Intersci., 1972.

\bibitem{CH} R. Courant, D. Hilbert, {\it Methods of Mathematical Physics.}
Vol.~\textbf{1}. Interscience, 1953.

\bibitem{FK} D. W. Fox, J. R. Kuttler, Sloshing frequencies. {\it Z. angew. Math.
Phys.} \textbf{34} (1983), 668--696.

\bibitem{John1} F. John, On the motion of floating bodies, II. {\it Comm. Pure
Appl. Math.} {\bf 3} (1950), 45--101.

\bibitem{KK} N. D. Kopachevsky, S. G. Krein, {\it Operator Approach to Linear
Problems of Hydrodynamics.} Vol.~\textbf{1}. Birkh\"auser Verlag, Basel, 2001.

\bibitem{KKM} V. Kozlov, N. Kuznetsov,  O. Motygin, On the two-dimensional sloshing
problem. {\it Proc. R. Soc. Lond. A} \textbf{460} (2004), 2587--2603.

\bibitem{KKM1} V. Kozlov, N. Kuznetsov,  O. Motygin, On the two-dimensional sloshing
problem. Correction. {\it Proc. R. Soc. Lond. A} \textbf{467} (2011), 2427--2430.

\bibitem{Kut} J. R. Kuttler, A nodal line theorem for the sloshing problem. {\it
SIAM J. Math. Anal.} \textbf{15} (1984), 1234--1237.

\bibitem{KM} N. G. Kuznetsov, O. V. Motygin The Steklov problem in a half-plane: the
dependence of eigenvalues on a piecewise constant coefficient. {\it J. Math. Sci.}
\textbf{127} (2005), 2429--2445.

\bibitem{SS} M. Levitin, L. Parnovski, I. Polterovich, D. A. Sher, Sloshing, Steklov
and corners: Asymptotics of sloshing eigenvalues. {\it J. Anal. Math.} \textbf{146}
(2022), 65--125.

\bibitem{W} P. Weidman, Analytical solutions of first-mode sloshing in new
containers. {\it Wave Motion} \textbf{63} (2016), 170--178.


}
\end{thebibliography}
\end{document}